\documentclass[12pt]{article} 
\usepackage{amsfonts,amsmath,latexsym,amssymb,mathrsfs,amsthm}
\usepackage{slashbox}
\usepackage{caption}

\let\OLDthebibliography\thebibliography
\renewcommand\thebibliography[1]{
  \OLDthebibliography{#1}
  \setlength{\parskip}{3pt}
  \setlength{\itemsep}{0pt plus 0.3ex}
}

\def\numberlikeadb{\global\def\theequation{\thesection.\arabic{equation}}}
\numberlikeadb
\newtheorem{theorem}{Theorem}[section]
\newtheorem{lemma}[theorem]{Lemma}
\newtheorem{open}[theorem]{Open Problem}
\newtheorem{corollary}[theorem]{Corollary}

\newtheorem{proposition}[theorem]{Proposition}
\newtheorem{remark}[theorem]{Remark}

\usepackage{lscape}
\usepackage{caption}
\usepackage{multirow}
\begin{document}

\title{Inequalities for integrals of modified Bessel functions and expressions involving them}
\author{Robert E. Gaunt\footnote{School of Mathematics, The University of Manchester, Manchester M13 9PL, UK}}

\date{\today} 
\maketitle

\vspace{-5mm}

\begin{abstract}Simple inequalities are established for some integrals involving the  modified Bessel functions of the first and second kind.  In most cases, we show that we obtain the best possible constant or that our bounds are tight in certain limits.  We apply these inequalities to obtain uniform bounds for several expressions involving integrals of modified Bessel functions.  Such expressions occur in Stein's method for variance-gamma approximation, and the results obtained in this paper allow for technical advances in the method.  We also present some open problems that arise from this research.  
\end{abstract}


\noindent{{{\bf{AMS 2010 Subject Classification:}}} Primary 33C10; 26D15

\section{Introduction}\label{intro}

\subsection{Motivation through Stein's method for variance-gamma approximation}

Stein's method \cite{stein} is a powerful technique in probability theory for deriving bounds for distributional approximations with respect to a probability metric, with applications in areas as diverse as random graph theory \cite{bhj92}, number theory \cite{harper} and random matrix theory \cite{fg15}.  The method is particularly well developed for normal approximation (see the books \cite{chen, np12}), and there is active research into extensions to non-normal limits; see the survey \cite{ross}.

Recently, Stein's method has been extended to variance-gamma (VG) approximation \cite{eichelsbacher, gaunt vg}.  The VG distribution (also known as the generalized Laplace distribution \cite{kkp01}) is commonly used in financial mathematics \cite{madan}, and has recently appeared in several papers in the probability literature as a limiting distribution \cite{aaps17, azmooden, bt17}.  This is in part due to the fact that the family of VG distributions is a rich one, with special or limiting cases that include, amongst others, the normal, gamma, Laplace, product of zero mean normals and difference of gammas \cite{gaunt vg, kkp01}.  It is therefore of interest to develop Stein's method for VG approximation to put some of the existing literature on Stein's method into a more general framework, and, more importantly, to extend it to new limit distributions.

At the heart of Stein's method for VG approximation is the function $f_h:\mathbb{R}\rightarrow\mathbb{R}$ defined by
\begin{align*}f_h(x)&=-\frac{\mathrm{e}^{-\beta x} K_{\nu}(|x|)}{|x|^{\nu}} \int_0^x \mathrm{e}^{\beta t} |t|^{\nu} I_{\nu}(|t|) h(t) \,\mathrm{d}t  -\frac{\mathrm{e}^{-\beta x} I_{\nu}(|x|)}{|x|^{\nu}} \int_x^{\infty} \mathrm{e}^{\beta t} |t|^{\nu} K_{\nu}(|t|)h(t)\,\mathrm{d}t, 
\end{align*}
where $\nu>-\frac{1}{2}$, $-1<\beta<1$, and $h:\mathbb{R}\rightarrow\mathbb{R}$ is smooth and such that $\mu(h)=0$, for $\mu$ the VG probability measure.  A crucial part of the method is to obtain uniform bounds, in terms of the supremum norms of $h$ and its derivatives,  for $f_h(x)$ and its first four derivatives.  In order to obtain these bounds, new inequalities were derived for integrals of modified Bessel functions \cite{gaunt ineq1, gaunt ineq2}, which were then used in the papers \cite{gaunt vg, dgv15} to bound derivatives of all order. 
%\begin{equation*}\tilde{h}(y)=h(y)-\int_{-\infty}^{\infty}\frac{(1-\beta^2)^{\nu+\frac{1}{2}}}{\sqrt{\pi}\Gamma(\nu+\frac{1}{2})2^{\nu}}\mathrm{e}^{\beta t} |t|^{\nu} K_{\nu}(|t|)h(t)\,\mathrm{d}t,
%\end{equation*}
%and $\nu>-\frac{1}{2}$, $-1<\beta<1$, and $h:\mathbb{R}\rightarrow\mathbb{R}$ is smooth.   

To obtain distributional approximations in stronger probability metrics (such as the Kolmogorov and Wasserstein metrics), alternative bounds for $f_h$ and its derivatives are required, which have a different dependence on the function $h$.  This is the focus of \cite{gaunt vg2, gaunt vg3}, and to achieve such bounds, new inequalities are required for certain expressions involving integrals of modified Bessel functions.  In this paper, we establish uniform bounds for some of these terms. In particular, we shall focus on bounding expressions of the type
%\begin{equation}\label{mlhwsb1} \frac{\mathrm{e}^{-\beta x}K_{\nu+1}(x)}{x^{\nu}}\int_0^{x} \mathrm{e}^{\beta t}t^{\nu+1} I_{\nu}(t)\,\mathrm{d}t, \qquad \frac{\mathrm{e}^{-\beta x}I_{\nu+1}(x)}{x^{\nu}} \int_x^{\infty} \mathrm{e}^{\beta t}t^{\nu+1} K_{\nu}(t)\,\mathrm{d}t 
%\end{equation}
%and
%\begin{equation}\label{mlhwsb2} \frac{\mathrm{e}^{-\beta x}K_{\nu+1}(x)}{x^{\nu-1}}\int_0^{x} \mathrm{e}^{\beta t}t^{\nu} I_{\nu}(t)\,\mathrm{d}t, \qquad \frac{\mathrm{e}^{-\beta x}I_{\nu+1}(x)}{x^{\nu-1}} \int_x^{\infty} \mathrm{e}^{\beta t}t^{\nu} K_{\nu}(t)\,\mathrm{d}t. 
%\end{equation}
\begin{align}\label{mlhwsb1} \frac{\mathrm{e}^{-\beta x}K_{\nu+1}(x)}{x^{\nu}}\int_0^{x} \mathrm{e}^{\beta t}t^{\nu+1} I_{\nu}(t)\,\mathrm{d}t,& \qquad \frac{\mathrm{e}^{-\beta x}I_{\nu+1}(x)}{x^{\nu}} \int_x^{\infty} \mathrm{e}^{\beta t}t^{\nu+1} K_{\nu}(t)\,\mathrm{d}t, \\
\label{mlhwsb2} \frac{\mathrm{e}^{-\beta x}K_{\nu+1}(x)}{x^{\nu-1}}\int_0^{x} \mathrm{e}^{\beta t}t^{\nu} I_{\nu}(t)\,\mathrm{d}t,& \qquad \frac{\mathrm{e}^{-\beta x}I_{\nu+1}(x)}{x^{\nu-1}} \int_x^{\infty} \mathrm{e}^{\beta t}t^{\nu} K_{\nu}(t)\,\mathrm{d}t. 
\end{align}
In \cite{gaunt vg2, gaunt vg3}, these bounds are used in the development of a framework for deriving Kolmogorov and Wasserstein distance error bounds for VG approximation via Stein's method.  The case $\beta=0$ is dealt with in \cite{gaunt vg2} and the case $\beta\not=0$ will be dealt with in \cite{gaunt vg3}.  In \cite{gaunt vg2}, this framework is applied to obtain explicit bounds for VG approximation in a variety of settings, including quantitative six moment theorems for the VG approximation of double Wiener-It\^{o} integrals (see \cite{eichelsbacher} for related results); VG approximation for a special special case of the $D_2$ statistic for alignment-free sequence comparison \cite{bla, lippert}; and Laplace approximation of a random sum of independent mean zero random variables (see \cite{pike} for related results).  Further applications will be given in the companion paper \cite{gaunt vg3}.

\subsection{Summary of the paper}

The approach we shall take to bounding these expressions is to first bound the integrals in (\ref{mlhwsb1}) and (\ref{mlhwsb2}).  Closed form expressions for these integrals, in terms of modified Bessel functions and the modified Struve function $\mathbf{L}_{\nu}(x)$, do in fact exist if $\beta=0$.  In this case, the integrals in (\ref{mlhwsb1}) take a very simple form (see (\ref{diffone}) and (\ref{diffKi})).  For $x>0$ and $\nu>-\frac{1}{2}$, let $\mathscr{L}_{\nu}(x)$ denote $I_{\nu}(x)$, $ \mathrm{e}^{\nu \pi i}K_{\nu}(x)$ or any linear combination of these functions, in which the coefficients are independent of $\nu$ and $x$.  From formula 10.43.2 of \cite{olver},
\begin{equation}\label{besint6}\int x^{\nu}\mathscr{L}_{\nu}(x)\,\mathrm{d}x =\sqrt{\pi}2^{\nu-1}\Gamma(\nu+\tfrac{1}{2})x\big(\mathscr{L}_{\nu}(x)\mathbf{L}_{\nu-1}(x)-\mathscr{L}_{\nu-1}(x)\mathbf{L}_{\nu}(x)\big). 
\end{equation}
There are no closed form expressions in terms of modified Bessel and Struve functions in the literature for the integrals in (\ref{mlhwsb1}) and (\ref{mlhwsb2}) for the case $\beta\not=0$.  Moreover, even when $\beta=0$ the expression on the right-hand side of (\ref{besint6}) is a complicated expression involving the modified Struve function $\mathbf{L}_{\nu}(x)$.  This provides the motivation for establishing simple bounds, in terms of modified Bessel functions, for the integrals given in (\ref{mlhwsb1}) and (\ref{mlhwsb2}).

In a recent work, \cite{gaunt ineq1} obtained simple inequalities involving modified Bessel functions for the integrals of (\ref{mlhwsb2}), which were used in \cite{gaunt ineq2} to bound a number of expressions that arise in Stein's method for VG approximation.  In Section \ref{sec2} of this paper, we obtain similar such bounds that will allow us to bound the expressions in (\ref{mlhwsb2}), and also obtain improvements on the inequalities of \cite{gaunt ineq1}.  Indeed, many of our bounds (see Theorems \ref{tiger}, \ref{tiger1} and Remark \ref{rrrr}) have the best possible constants or are tight in a certain limit.  We shall also obtain inequalities for the integrals of (\ref{mlhwsb1}), which, to best knowledge of this author, have not previously been studied.  The integral inequalities obtained in this paper shall have an immediate application to Stein's method for VG approximation.  The bounds may also prove to be useful in other problems involving modified Bessel functions; see for example, \cite{baricz3} in which inequalities for modified Bessel functions of the first kind were used to obtain lower and upper bounds for integrals involving modified Bessel functions of the first kind.

In Section \ref{sec3}, the integral inequalities that are derived in Section \ref{sec2} are applied, together with known inequalities for products of modified Bessel functions, to obtain uniform bounds for the expressions in (\ref{mlhwsb1}) and (\ref{mlhwsb2}).  We are able to establish these bounds for the whole parameter range $\nu>-\frac{1}{2}$ and $-1<\beta<1$, except for the first expression of (\ref{mlhwsb2}).  Straightforward calculations using the limiting forms of Section \ref{asymsec} confirm that the expression is bounded for all $x>0$ in the whole parameter range; however, deriving an explicit upper bound in terms of $\nu$ and $\beta$ becomes difficult if $\nu<\frac{1}{2}$ and $\beta<0$.  We make some partial progress (see Theorem \ref{openthm}), but we leave this as an open problem (see Open Problems \ref{openprob1}, \ref{openprob2} and \ref{openprob3}).  In the Appendix, we state some elementary properties of modified Bessel functions that are used throughout this paper.

%To obtaining our bounds, we need to draw on a number of results from the literature on inequalities for modified Bessel functions and their integrals, which we present in Section 2.  Further elementary properties of modified Bessel functions that are used throughout this paper are given in the Appendix.  In Section 3, we obtain upper bounds for the integrals on the right-hand side of (\ref{mlhwsb1}) and (\ref{mlhwsb2}).  

% We shall bound (\ref{mlhwsb1}) for all $\nu>-\frac{1}{2}$ and $-1<\beta<1$ (see Theorem \ref{vole}), but will only bound (\ref{mlhwsb2}) for the case $n=1$ (see Theorem \ref{appd1}).  We shall then move on to consider the case $\beta=0$.  For this case, we obtain improved bounds for (\ref{mlhwsb1}) (see Theorem \ref{ml7min}) and are able to bound (\ref{mlhwsb2}) for all $n\geq 1$ (see Theorem \ref{ivy}).

%The rest of the paper is organised as follows.  In Section 2, we state a number of results from the literature on inequalities for modified Bessel functions and their integrals that we will be needed in our proofs.  In Section 3, we obtain upper bounds for the integrals given in the display (\ref{mlhwsb1}); several inequalities for the integrals from (\ref{mlhwsb2}) have already been obtained by \cite{gaunt ineq1}.  In Section 4, we obtain uniform bounds for the expressions involving modified Bessel functions of the type that have been presented in this section.  In the Appendix, we state some elementary properties of modified Bessel functions that are used throughout this paper.

\section{Inequalities for integrals involving modified Bessel functions}\label{sec2}

Our first proposition contains some results that are easy consequences of some of the inequalities of Theorems 2.1 and 2.5 of \cite{gaunt ineq1}.  As shall be the case with the following theorems of this section, the inequalities will be needed in Section \ref{sec3}.

\begin{proposition}\label{propone}Let $\beta\geq0$.  Then, for $x>0$,
\begin{eqnarray}\label{propb2a1}\int_0^x \mathrm{e}^{\beta t}t^{\nu}I_\nu(t)\,\mathrm{d}t&<& \frac{2(\nu+1)}{2\nu+1}\mathrm{e}^{\beta x}x^{\nu}I_{\nu+1}(x), \quad \nu>-\tfrac{1}{2}, \\
\label{propb2a}\int_0^x \mathrm{e}^{\beta t}t^{\nu+1}I_\nu(t)\,\mathrm{d}t&\leq& \mathrm{e}^{\beta x}x^{\nu+1}I_{\nu+1}(x),\quad \nu>-1, 
\end{eqnarray}
\begin{eqnarray}
\label{fff1}\int_x^\infty \mathrm{e}^{-\beta t}t^{\nu}K_{\nu}(t)\,\mathrm{d}t&<& \mathrm{e}^{-\beta x}x^{\nu}K_{\nu+1}(x), \quad\nu\in\mathbb{R}, \\
\label{fff}\int_x^\infty \mathrm{e}^{-\beta t}t^{\nu+1}K_{\nu}(t)\,\mathrm{d}t&\leq& \mathrm{e}^{-\beta x}x^{\nu+1}K_{\nu+1}(x), \quad \nu\in\mathbb{R}.
\end{eqnarray}
We have equality in (\ref{propb2a}) and (\ref{fff}) if and only if $\beta=0$.
\end{proposition}

\begin{proof}Since $\beta\geq0$, the function $\mathrm{e}^{\beta t}$ is non-decreasing in $t$, and therefore $\int_0^x \mathrm{e}^{\beta t}t^{\nu}I_\nu(t)\,\mathrm{d}t\leq \mathrm{e}^{\beta x}\int_0^x t^{\nu}I_\nu(t)\,\mathrm{d}t$, where we used that $I_\nu(x)>0$ for all $x>0$ if $\nu>-1$.  We can use the strict inequality $\int_0^x t^{\nu}I_\nu(t)\,\mathrm{d}t< \frac{2(\nu+1)}{2\nu+1}x^{\nu}I_{\nu+1}(x)$, which is valid for $\nu>-\frac{1}{2}$, from Theorem 2.1 of \cite{gaunt ineq1} to obtain  (\ref{propb2a1}).  The proof of inequality (\ref{propb2a}) is similar, but the integral can now be evaluated using (\ref{diffone}).  Note that the integral exists if $\nu>-1$.  Inequality (\ref{fff1}) follows from using the inequality $\int_x^\infty \mathrm{e}^{-\beta t}t^{\nu}K_{\nu}(t)\,\mathrm{d}t<\mathrm{e}^{-\beta x}\int_x^\infty t^{\nu}K_{\nu}(t)\,\mathrm{d}t$ and then bounding the integral using the inequality $\int_x^\infty t^{\nu}K_{\nu}(t)\,\mathrm{d}t<x^{\nu}K_{\nu+1}(x)$ of Theorem 2.5 of \cite{gaunt ineq1}.  Finally, the proof of (\ref{fff}) is a similar, but the integral can now be evaluated using (\ref{diffKi}).
\end{proof}

The integral inequalities obtained in the remainder of this section will require more careful arguments, because, in most cases, we will no longer be able to remove the exponentials $\mathrm{e}^{\beta t}$ and $\mathrm{e}^{-\beta t}$ from the integrals.  Before stating the following theorem, it will be useful to introduce some notation.  Let $\nu>-\frac{1}{2}$ and $-1<\beta<1$.  Then we define
\begin{equation*}I_{\nu,\beta}:=\int_0^\infty \mathrm{e}^{\beta t}t^\nu K_\nu(t)\,\mathrm{d}t.
\end{equation*}
An explicit formula for the integral, involving the Ferrers function $P$, is given in formula 10.43.22 of \cite{olver}.  However, for our purposes, it is more useful to note the double inequality
\begin{equation}\label{doubleivb}\frac{\sqrt{\pi}\Gamma(\nu+\frac{1}{2})2^{\nu-1}}{(1-\beta^2)^{\nu+\frac{1}{2}}}\leq I_{\nu,\beta}<\frac{\sqrt{\pi}\Gamma(\nu+\frac{1}{2})2^{\nu}}{(1-\beta^2)^{\nu+\frac{1}{2}}},
\end{equation}
where we have equality in the lower bound if and only if $\beta=0$.  These inequalities are an easy consequence of (\ref{pdfk}).

With this notation, we state our theorem.  Parts (i) and (ii) are an improvement on some inequalities from Theorem 2.5 of \cite{gaunt ineq1}.

\begin{theorem}\label{tiger}Let $0<\beta<1$.  Then

(i) For all $x>0$,
\begin{equation}\label{lowerk}\int_x^\infty \mathrm{e}^{\beta t}t^\nu K_\nu(t)\,\mathrm{d}t\leq\frac{1}{1-\beta}\mathrm{e}^{\beta x}x^\nu K_{\nu}(x), \quad \nu\leq\tfrac{1}{2}.
\end{equation}
We have equality in (\ref{lowerk}) if and only if $\nu=\frac{1}{2}$, and the inequality is reversed if $\nu>\frac{1}{2}$.

(ii) Let $\nu>\frac{1}{2}$.  Then the following double inequality holds for all $x>0$,
\begin{equation}\label{lowerk2}\frac{1}{1-\beta}\mathrm{e}^{\beta x}x^\nu K_{\nu}(x)<\int_x^\infty \mathrm{e}^{\beta t}t^\nu K_\nu(t)\,\mathrm{d}t<\frac{I_{\nu,\beta}}{2^{\nu-1}\Gamma(\nu)}\mathrm{e}^{\beta x}x^\nu K_{\nu}(x).
\end{equation}
The constants in the upper and lower bounds of (\ref{lowerk2}) cannot be improved.

(iii) Now let $\nu\geq-\frac{1}{2}$.  Then the following double inequality holds for all $x>0$,
\begin{equation}\label{lowerk3}\frac{1}{1-\beta}\mathrm{e}^{\beta x}x^{\nu+1} K_{\nu+1}(x)\leq\int_x^\infty \mathrm{e}^{\beta t}t^{\nu+1} K_\nu(t)\,\mathrm{d}t<\bigg(1+\frac{\beta I_{\nu+1,\beta}}{2^\nu \Gamma(\nu+1)}\bigg)\mathrm{e}^{\beta x}x^{\nu+1} K_{\nu+1}(x).
\end{equation}
We have equality in the lower bound if and only if $\nu=-\frac{1}{2}$.  The constants in the upper and lower bounds of (\ref{lowerk3}) cannot be improved.  
\end{theorem}

\begin{proof}(i) Suppose that $\nu< \frac{1}{2}$.  On using integration by parts and the differentiation formula (\ref{diffKi}), we obtain
\begin{equation*}\int_x^{\infty}\mathrm{e}^{\beta t}t^{\nu}K_{\nu}(t)\,\mathrm{d}t =-\frac{1}{\beta}\mathrm{e}^{\beta x}x^{\nu}K_{\nu}(x) +\frac{1}{\beta}\int_x^{\infty}\mathrm{e}^{\beta t} t^{\nu}K_{\nu-1}(t)\,\mathrm{d}t.
\end{equation*}
Applying the inequality (\ref{Kmoni}) and rearranging gives
\[\left(\frac{1}{\beta}-1\right)\int_x^{\infty}\mathrm{e}^{\beta t}t^{\nu}K_{\nu}(t)\,\mathrm{d}t< \frac{1}{\beta}\mathrm{e}^{\beta x}x^{\nu}K_{\nu}(x).\]
Inequality (\ref{lowerk}) now follows on rearranging.  If $\nu=\frac{1}{2}$, then we have equality because $K_{-\frac{1}{2}}(x)=K_{\frac{1}{2}}(x)$.  Finally, if $\nu>\frac{1}{2}$, then applying inequality (\ref{cake}) instead of (\ref{Kmoni}) reverses inequality (\ref{lowerk}).

(ii) The lower bound was established in part (i).  The upper bound represents an improvement in the constant of the final bound of Theorem 2.5 of \cite{gaunt ineq1}.  Our proof uses the same approach as \cite{gaunt ineq1} but involves a more careful analysis to ensure we obtain the best possible constant.  Define the function
\begin{equation*}v(x)=M_{\nu,\beta}\mathrm{e}^{\beta x}x^\nu K_{\nu}(x)-\int_x^\infty \mathrm{e}^{\beta t}t^\nu K_\nu(t)\,\mathrm{d}t,
\end{equation*} 
where $M_{\nu,\beta}=\frac{I_{\nu,\beta}}{2^{\nu-1}\Gamma(\nu)}$. Then proving that $v(x)>0$ for all $x>0$ establishes (\ref{lowerk2}).  We begin by noting that $v(0)=0$ and $\lim\nolimits_{x\rightarrow \infty}v(x)=0$, which are verified by the following calculations, where we make use of the limiting forms (\ref{Ktend0}) and (\ref{Ktendinfinity}):
\begin{align*}v(0)&=\frac{I_{\nu,\beta}}{2^{\nu-1}\Gamma(\nu)}\lim_{x\downarrow0}\mathrm{e}^{\beta x}x^\nu K_{\nu}(x)-\int_0^{\infty}  \mathrm{e}^{\beta t}t^{\nu} K_{\nu}(t)\,\mathrm{d}t \\
&=\frac{I_{\nu,\beta}}{2^{\nu-1}\Gamma(\nu)}\cdot 2^{\nu-1}\Gamma(\nu) -I_{\nu,\beta}=0,
\end{align*}
and
\begin{equation*}\lim_{x\rightarrow \infty}v(x)=\lim_{x\to \infty}M_{\nu,\beta}\mathrm{e}^{\beta x}  x^{\nu} K_{\nu}(x)-\lim_{x\to \infty}\int_x^{\infty}  \mathrm{e}^{\beta t}t^{\nu} K_{\nu}(t)\,\mathrm{d}t =0.
\end{equation*}
Now, we analyse the derivative of $v(x)$.  On using the differentiation formula (\ref{diffKi}) we obtain
\begin{equation}\label{uder1}v'(x)=\mathrm{e}^{\beta x}x^{\nu} [(1+\beta M_{\nu,\beta})K_{\nu}(x) -M_{\nu,\beta}K_{\nu-1}(x)].
\end{equation}
Based on expression (\ref{uder1}), we can argue that $v(x)>0$ for all $x>0$.  To see, this we note that for a given $\nu>\frac{1}{2}$ and $0<\beta<1$, we have either $1+\beta M_{\nu,\beta}\geq M_{\nu,\beta}$ or $1+\beta M_{\nu,\beta}<M_{\nu,\beta}$.  In the former case, as $K_\nu(x)>K_{\nu-1}(x)$ for all $x>0$ if $\nu>\frac{1}{2}$, then $v'(x)$ would be strictly positive for all $x>0$, and since $v(0)=0$ it would follow that $v(x)>0$ for all $x>0$.  (Note that this would contradict $v(0)=0$ and $\lim_{x\rightarrow\infty}v(x)=0$.)  In the latter case, we note that in the limit $x\downarrow0$ we have, by the limiting form (\ref{Ktend0}), that
\begin{equation*}v'(x) \sim (1+\beta M_{\nu,\beta})2^{\nu-1}\Gamma(\nu).
\end{equation*}
Hence, $v(x)$ is initially an increasing function of $x$.  Now, Corollary 2.3 of \cite{gaunt ineq1} tells us that for $\nu>\frac{1}{2}$ and $\alpha>1$ the equation $K_\nu(x)=\alpha K_{\nu-1}(x)$ has exactly one root in the region $x>0$.  Since $1+\beta M_{\nu,\beta}<M_{\nu,\beta}$, it follows that (\ref{uder1}) has exactly one root $x^*$ in the region $x>0$.  Putting everything together, we see that $v(x)$ takes the value 0 at the origin before monotonically increasing to a maximum value which takes place at $x=x^*$ and then decreases monotonically down to 0 in the limit $x\rightarrow\infty$.  Therefore, we conclude that $v(x)>0$ for all $x>0$, as required.   

Since $v(0)=0$, the constant in the upper bound of (\ref{lowerk2}) cannot be improved.  To establish that the constant in the lower bound cannot be improved, we obtain a limiting form for the integral.  Using (\ref{Ktendinfinity}) and integration by parts, we obtain, as $x\rightarrow\infty$,
\begin{align*}\int_x^\infty \mathrm{e}^{\beta t}t^{\nu} K_\nu(t)\,\mathrm{d}t&\sim \int_x^\infty \mathrm{e}^{\beta t}t^{\nu} \cdot \sqrt{\frac{\pi}{2t}}\mathrm{e}^{-t}\,\mathrm{d}t \\
&=\sqrt{\frac{\pi}{2}}\int_x^\infty \mathrm{e}^{-(1-\beta) t}t^{\nu-\frac{1}{2}}\,\mathrm{d}t \\
&=\sqrt{\frac{\pi}{2}}\frac{1}{1-\beta}\mathrm{e}^{-(1-\beta) x}x^{\nu-\frac{1}{2}}+\sqrt{\frac{\pi}{2}}\frac{\nu-\frac{1}{2}}{1-\beta}\int_x^\infty \mathrm{e}^{-(1-\beta) t}t^{\nu-\frac{3}{2}}\,\mathrm{d}t.
\end{align*}
But $\int_x^\infty \mathrm{e}^{-(1-\beta) t}t^{\nu-\frac{3}{2}}\,\mathrm{d}t\ll \int_x^\infty \mathrm{e}^{-(1-\beta) t}t^{\nu-\frac{1}{2}}\,\mathrm{d}t$, as $x\rightarrow\infty$, and so
\begin{equation}\label{zon1}\int_x^\infty \mathrm{e}^{\beta t}t^{\nu} K_\nu(t)\,\mathrm{d}t\sim \sqrt{\frac{\pi}{2}}\frac{1}{1-\beta}\mathrm{e}^{-(1-\beta) x}x^{\nu-\frac{1}{2}}, \quad x\rightarrow\infty. 
\end{equation}
Also, 
\begin{equation}\label{zon2}\frac{1}{1-\beta}\mathrm{e}^{\beta x}x^\nu K_{\nu}(x)\sim\sqrt{\frac{\pi}{2}}\frac{1}{1-\beta}\mathrm{e}^{-(1-\beta) x}x^{\nu-\frac{1}{2}}, \quad x\rightarrow\infty.
\end{equation} 
The equivalence between (\ref{zon1}) and (\ref{zon2}) confirms that the constant in the lower bound of (\ref{lowerk2}) cannot be improved. 

(iii) Now let $\nu\geq-\frac{1}{2}$.  Using integration by parts and the differentiation formula (\ref{diffKi}) gives that
\begin{align}\label{zoonu}\int_x^\infty \mathrm{e}^{\beta t}t^{\nu+1} K_\nu(t)\,\mathrm{d}t&=\mathrm{e}^{\beta x}x^{\nu+1}K_{\nu+1}(x)+\beta\int_x^\infty \mathrm{e}^{\beta t}t^{\nu+1} K_{\nu+1}(t)\,\mathrm{d}t \\
&\geq \mathrm{e}^{\beta x}x^{\nu+1}K_{\nu+1}(x)+\beta\int_x^\infty \mathrm{e}^{\beta t}t^{\nu+1} K_{\nu}(t)\,\mathrm{d}t, \nonumber
\end{align}
where we used (\ref{cake}) to obtain the inequality.  Since $K_{-\frac{1}{2}}(x)=K_{\frac{1}{2}}(x)$, we have equality when $\nu=-\frac{1}{2}$.  Rearranging yields the lower bound of (\ref{lowerk3}), as required.  The upper bound of (\ref{lowerk3}) follows from applying the upper bound of (\ref{lowerk2}) to the integral on the right-hand side of (\ref{zoonu}).

That the constant in the lower bound cannot be improved follows from the same argument that was used in part (ii).  For the upper bound, we note that, on the one hand,
\begin{align*}\int_0^\infty \mathrm{e}^{\beta t}t^{\nu+1} K_\nu(t)\,\mathrm{d}t&=\lim_{x\downarrow0}\mathrm{e}^{\beta x}x^{\nu+1}K_{\nu+1}(x)+\beta\int_0^\infty \mathrm{e}^{\beta t}t^{\nu+1} K_{\nu+1}(t)\,\mathrm{d}t \\
&=2^\nu\Gamma(\nu+1)+\beta I_{\nu+1,\beta},
\end{align*}
and on the other,
\begin{align*}\bigg(1+\frac{\beta I_{\nu+1,\beta}}{2^\nu \Gamma(\nu+1)}\bigg)\lim_{x\downarrow0}\mathrm{e}^{\beta x}x^{\nu+1} K_{\nu+1}(x)&=\bigg(1+\frac{\beta I_{\nu+1,\beta}}{2^\nu \Gamma(\nu+1)}\bigg)\cdot 2^\nu\Gamma(\nu+1)\\
&=2^\nu\Gamma(\nu+1)+\beta I_{\nu+1,\beta},
\end{align*}
which confirms that the constant cannot be improved.  The proof is complete.
\end{proof}

In the following theorem, inequalities (\ref{besi22}) and (\ref{besi33}) represent improvements on inequalities from Theorem 2.1 of \cite{gaunt ineq1}, whilst the final two inequalities are, to the best knowledge of this author, original in the literature.

\begin{theorem}\label{tiger1}Let $0<\gamma<1$ and $n>-1$. Then, for all $x>0$,
\begin{align}\label{besi11}\int_0^x \mathrm{e}^{-\gamma t}t^\nu I_{\nu+n}(t)\,\mathrm{d}t&>\mathrm{e}^{-\gamma x}x^\nu I_{\nu+n+1}(x), \quad\nu>-\tfrac{1}{2}(n+1),  \\
\label{besi22}\int_0^x t^\nu I_{\nu+n}(t)\,\mathrm{d}t&<\frac{x^\nu}{2\nu+n+1}\Big(2(\nu+n+1)I_{\nu+n+1}(x)-(n+1)I_{\nu+n+3}(x)\Big), \\
\label{besi225}&<\frac{2(\nu+n+1)}{2\nu+n+1}x^\nu I_{\nu+n+1}(x), \quad  \nu>-\tfrac{1}{2}(n+1), \\
\label{besi33}\int_0^x \mathrm{e}^{-\gamma t}t^{\nu}I_\nu(t)\,\mathrm{d}t&<\frac{\mathrm{e}^{-\gamma x}x^\nu}{(2\nu+1)(1-\gamma)}\Big(2(\nu+1)I_{\nu+1}(x)-I_{\nu+3}(x)\Big), \quad \nu\geq\tfrac{1}{2}, \\
\label{besi44}\int_0^x \mathrm{e}^{-\gamma t}t^{\nu+1} I_{\nu}(t)\,\mathrm{d}t&>\mathrm{e}^{-\gamma x}x^{\nu+1} I_{\nu+1}(x), \quad \nu>-1, \\
\label{besi55}\int_0^x \mathrm{e}^{-\gamma t}t^{\nu+1} I_{\nu}(t)\,\mathrm{d}t&<\frac{1}{1-\gamma}\mathrm{e}^{-\gamma x}x^{\nu+1} I_{\nu+1}(x), \quad\nu>-\tfrac{1}{2}.
\end{align}
The constants in the bounds (\ref{besi225}), (\ref{besi44}) and (\ref{besi55}) cannot be improved.  Inequalities (\ref{besi11}) and (\ref{besi44}) hold for all $\gamma>0$.  
\end{theorem}

\begin{proof}(i) The condition $\nu>-\tfrac{1}{2}(n+1)$ ensures that the integral exists. As $\gamma>0$ and $n>-1$, on using the differentiation formula (\ref{diffone}) we have that
\begin{align*}\int_0^x\mathrm{e}^{-\gamma t}t^{\nu}I_{\nu+n}(t)\,\mathrm{d}t &=\int_0^x\mathrm{e}^{-\gamma t}\frac{1}{t^{n+1}}t^{\nu+n+1}I_{\nu+n}(t)\,\mathrm{d}t\\
& >\frac{\mathrm{e}^{-\gamma x}}{x^{n+1}}\int_0^xt^{\nu+n+1}I_{\nu+n}(t)\,\mathrm{d}t =\mathrm{e}^{-\gamma x}x^{\nu}I_{\nu+n+1}(x),
\end{align*}
since by (\ref{Itend0}) we have $\lim_{x\downarrow 0}x^{\nu+n+1}I_{\nu+n+1}(x)=0$ if $n>-1$ and $\nu>-\tfrac{1}{2}(n+1)$.

(ii) Inequality (\ref{besi22}) improves inequality (2.6) of Theorem 2.1 of \cite{gaunt ineq1}.  To obtain this improvement, we follow the approach used to obtain that inequality, but argue more carefully.  A straightforward calculation, as given in \cite{gaunt ineq1}, using the differentiation formula (\ref{diffone}) and identity (\ref{Iidentity}) gives that
\begin{equation}\label{gafor} \frac{\mathrm{d}}{\mathrm{d}t} \big(t^{\nu} I_{\nu +n+1} (t)\big) = \frac{2\nu +n+1}{2(\nu +n+1)} t^{\nu} I_{\nu +n} (t) + \frac{n+1}{2(\nu +n+1)} t^{\nu} I_{\nu +n+2} (t). 
\end{equation}
Integrating both sides over $(0,x)$, applying the fundamental theorem of calculus and rearranging gives
\[\int_0^x t^{\nu} I_{\nu +n} (t)\,\mathrm{d}t = \frac{2(\nu +n+1)}{2\nu +n+1} x^{\nu} I_{\nu +n+1} (x) - \frac{n+1}{2\nu +n+1} \int_0^x t^{\nu} I_{\nu +n +2} (t)\,\mathrm{d}t.\]
Applying inequality (\ref{besi11}) with $\gamma=0$ to the integral on the right hand-side of the above expression gives that
\begin{align*}\int_0^x t^{\nu}I_{\nu+n}(t)\,\mathrm{d}t&<\frac{2(\nu+n+1)}{2\nu+n+1}x^\nu I_{\nu+n+1}(x)-\frac{n+1}{2\nu+n+1}x^{\nu}I_{\nu+n+3}(x) \\
&=\frac{x^\nu}{2\nu+n+1}\Big(2(\nu+n+1)I_{\nu+n+1}(x)-(n+1)I_{\nu+n+3}(x)\Big).
\end{align*}
Inequality (\ref{besi225}) (which is inequality (2.6) of Theorem 2.1 of \cite{gaunt ineq1}) follows from the fact that $I_{\nu+n+3}(x)>0$ for all $x>0$.  

We now prove that the constant in inequality (\ref{besi225}) cannot be improved.  From (\ref{Itend0}), we have on the one hand, as $x\downarrow0$,
\begin{equation}\label{form1}\int_0^x t^\nu I_{\nu+n}(t)\,\mathrm{d}t\sim\int_0^x \frac{t^{2\nu+n}}{2^{\nu+n}\Gamma(\nu+n+1)}\,\mathrm{d}t=\frac{x^{2\nu+n+1}}{2^{\nu+n}(2\nu+n+1)\Gamma(\nu+n+1)},
\end{equation}
and on the other,
\begin{align}\frac{2(\nu+n+1)}{2\nu+n+1}x^\nu I_{\nu+n+1}(x)&\sim \frac{2(\nu+n+1)}{2\nu+n+1}\frac{x^{2\nu+n+1}}{2^{\nu+n+1}\Gamma(\nu+n+2)}\nonumber\\
\label{form2}&=\frac{x^{2\nu+n+1}}{2^{\nu+n}(2\nu+n+1)\Gamma(\nu+n+1)},
\end{align}
which proves the claim.

(iii) Let $\nu\geq\frac{1}{2}$.  Then, a special case of inequality (2.5) of Theorem 2.5 of \cite{gaunt ineq1} states that, for all $x>0$,
\begin{equation}\label{gauntpa}\int_0^x \mathrm{e}^{-\gamma t}t^{\nu}I_\nu(t)\,\mathrm{d}t<\frac{\mathrm{e}^{-\gamma x}}{1-\gamma}\int_0^x t^{\nu}I_\nu(t)\,\mathrm{d}t.
\end{equation}
Bounding the integral on the right-hand side of (\ref{gauntpa}) using (\ref{besi22}) (with $n=0$) then yields inequality (\ref{besi33}), as required.

(iv) Let $\nu>-1$ so that the integral exists. Since $\gamma>0$,
\begin{equation*}\int_0^x \mathrm{e}^{-\gamma t}t^{\nu+1}I_\nu(t)\,\mathrm{d}t>\mathrm{e}^{-\gamma x}\int_0^x t^{\nu+1}I_\nu(t)\,\mathrm{d}t=\mathrm{e}^{-\gamma x}x^{\nu+1}I_{\nu+1}(x),
\end{equation*}
where we used (\ref{diffKi}) to evaluate the integral.

(v) Consider the function
\begin{equation*}u(x)=\frac{1}{1-\gamma}\mathrm{e}^{-\gamma x}x^{\nu+1}I_{\nu+1}(x)-\int_0^x\mathrm{e}^{-\gamma t}t^{\nu+1}I_\nu(t)\,\mathrm{d}t.
\end{equation*}
We argue that that $u(x)>0$ for all $x>0$, which will prove the result.  Using the differentiation formula (\ref{diffone}) we have that
\begin{align*}u'(x)&=\frac{1}{1-\gamma}\mathrm{e}^{-\gamma x}x^{\nu+1}\big(I_\nu(x)-\gamma I_{\nu+1}(x)\big)-\mathrm{e}^{-\gamma x}x^{\nu+1}I_\nu(x) \\
&=\frac{1}{1-\gamma}\mathrm{e}^{-\gamma x}x^{\nu+1}\big(I_\nu(x)-I_{\nu+1}(x)\big)>0,
\end{align*}
where we used (\ref{Imon}) to obtain the inequality.  Also, from (\ref{Itend0}), as $x\downarrow0$,
\begin{align*}u(x)&\sim \frac{1}{1-\gamma}\frac{x^{2\nu+2}}{\Gamma(\nu+2)2^{\nu+1}}-\int_0^x \frac{t^{2\nu+1}}{\Gamma(\nu+1)2^{\nu}}\,\mathrm{d}t\\
&=\frac{1}{1-\gamma}\frac{x^{2\nu+2}}{\Gamma(\nu+2)2^{\nu+1}}-\frac{x^{2\nu+2}}{2(\nu+1)\cdot\Gamma(\nu+1)2^{\nu}} \\
& =\frac{1}{1-\gamma}\frac{x^{2\nu+2}}{\Gamma(\nu+2)2^{\nu+1}}-\frac{x^{2\nu+2}}{\Gamma(\nu+2)2^{\nu+1}}=\frac{\gamma}{1-\gamma}\frac{x^{2\nu+2}}{\Gamma(\nu+2)2^{\nu+1}}>0.
\end{align*}
Thus, we conclude that $u(x)>0$ for all $x>0$, as required.

(vi) Finally, we prove that the constants in the bounds (\ref{besi44}) and (\ref{besi55}) cannot be improved.  Let $M>0$ and define
\begin{equation*}u_M(x)=M\mathrm{e}^{-\gamma x}x^{\nu+1}I_{\nu+1}(x)-\int_0^x\mathrm{e}^{-\gamma t}t^{\nu+1}I_\nu(t)\,\mathrm{d}t.
\end{equation*}
From an almost identical argument to the one used in part (v), we have that, as $x\downarrow0$,
\begin{equation*}u_M(x)\sim (M-1)\frac{x^{2\nu+2}}{\Gamma(\nu+2)2^{\nu+1}}.
\end{equation*}
Hence, if $M>1$ then $u_M(x)>0$ in a small positive neighbourhood of the origin, from which we conclude that the constant ($M=1$) in (\ref{besi44}) cannot be improved.

Let us now prove that the constant in the upper bound (\ref{besi55}) cannot be improved.  Again, arguing as we did in part (v), we have
\begin{equation*}u_M'(x)=\mathrm{e}^{-\gamma x}x^{\nu+1}\big(MI_\nu(x)-(M\gamma-1)I_{\nu+1}(x)\big).
\end{equation*}
From (\ref{Itendinfinity}), as $x\rightarrow\infty$,
\begin{equation*}u_M'(x)\sim \frac{1}{\sqrt{2\pi}}\big(M(1-\gamma)-1\big)x^{\nu+\frac{1}{2}}\mathrm{e}^{(1-\gamma)x}.
\end{equation*}
If $M<\frac{1}{1-\gamma}$, then $u_M'(x)$ decreases exponentially as $x\rightarrow\infty$, and hence there must exist a $x^*>0$ such that $u_M(x)<0$ for all $x\geq x^*$.  We therefore conclude that the constant ($M=\frac{1}{1-\gamma}$) in (\ref{besi55}) cannot be improved.  This completes the proof.
\end{proof}

\begin{remark}\label{rrrr}Let $n>-1$ and $\nu>-\frac{1}{2}(n+1)$.  Then from (\ref{besi11}) and (\ref{besi22}) we have the double inequality
\begin{align*}x^\nu I_{\nu+n+1}(x)<\int_0^x t^\nu &I_{\nu+n}(t)\,\mathrm{d}t <\frac{x^\nu}{2\nu+n+1}\Big(2(\nu+n+1)I_{\nu+n+1}(x)-(n+1)I_{\nu+n+3}(x)\Big).
\end{align*}
The double inequality is clearly tight as $\nu\rightarrow\infty$.  Also, on the one hand, as $x\rightarrow\infty$,
\begin{equation*}x^\nu I_{\nu+n+1}(x)\sim\frac{1}{\sqrt{2\pi}}x^{\nu-\frac{1}{2}}\mathrm{e}^{x},
\end{equation*}
and on the other,
\begin{equation*}\frac{x^\nu}{2\nu+n+1}\Big(2(\nu+n+1)I_{\nu+n+1}(x)-(n+1)I_{\nu+n+3}(x)\Big)\sim\frac{1}{\sqrt{2\pi}}x^{\nu-\frac{1}{2}}\mathrm{e}^{x},
\end{equation*}
from which we conclude that the double inequality is tight as $x\rightarrow\infty$.
\end{remark}

As noted by \cite{gaunt ineq1}, one can combine the inequalities of Theorems \ref{tiger} and \ref{tiger1} and the indefinite integral formula (\ref{besint6}) to obtain lower and upper bounds for the quantity $\mathscr{L}_{\nu}(x)\mathbf{L}_{\nu-1}(x)-\mathscr{L}_{\nu-1}(x)\mathbf{L}_{\nu}(x)$.  

\begin{corollary}\label{struvebessel}Let $\nu>\frac{1}{2}$.  Then, for all $x>0$,
\begin{equation}\label{dob11}\frac{x^{\nu}K_{\nu}(x)}{\sqrt{\pi}2^{\nu-1}\Gamma(\nu+\frac{1}{2})}<1-x\big(K_{\nu}(x)\mathbf{L}_{\nu-1}(x)+K_{\nu-1}(x)\mathbf{L}_{\nu}(x)\big)<\frac{x^{\nu}K_{\nu}(x)}{2^{\nu-1}\Gamma(\nu)}.
\end{equation}
Now let $\nu>-\frac{1}{2}$. Then, for all $x>0$,
\begin{align}\frac{x^{\nu-1}I_{\nu+1}(x)}{\sqrt{\pi}2^{\nu-1}\Gamma(\nu+\frac{1}{2})}<I_{\nu}(x)&\mathbf{L}_{\nu-1}(x)-I_{\nu-1}(x)\mathbf{L}_{\nu}(x) \nonumber \\
\label{dob22}&<\frac{x^{\nu-1}I_{\nu+1}(x)}{\sqrt{\pi}2^{\nu-1}\Gamma(\nu+\frac{1}{2})}\bigg\{1+\frac{1}{2\nu+1}\bigg(1-\frac{I_{\nu+3}(x)}{I_{\nu+1}(x)}\bigg)\bigg\}.
\end{align}
\end{corollary}

\begin{proof}From the limiting forms (\ref{Ktendinfinity}) and (\ref{struveinfinity}) for $K_{\nu}(x)$ and $\mathbf{L}(x)$, respectively, we have that
\[\lim_{x\rightarrow\infty}(x(K_{\nu}(x)\mathbf{L}_{\nu-1}(x)+K_{\nu-1}(x)\mathbf{L}_{\nu}(x)))=1, \quad \mbox{for $\nu>\tfrac{1}{2}$}.\]
Also, from the limiting forms (\ref{Itend0}) and (\ref{struve0}) for $I_{\nu}(x)$ and $\mathbf{L}(x)$, we have
\[\lim_{x\downarrow 0}(x(I_{\nu}(x)\mathbf{L}_{\nu-1}(x)-I_{\nu-1}(x)\mathbf{L}_{\nu}(x)))=0, \quad \mbox{for $\nu>-\tfrac{1}{2}$}.\]
Therefore, applying the indefinite integral formula (\ref{besint6}) gives
\begin{align}\label{chicot0}\int_x^\infty t^{\nu}K_{\nu}(t)\,\mathrm{d}t&=\sqrt{\pi}2^{\nu-1}\Gamma(\nu+\tfrac{1}{2})\big[1-x(K_{\nu}(x)\mathbf{L}_{\nu-1}(x)+K_{\nu-1}(x)\mathbf{L}_{\nu}(x))\big], \quad \nu>\tfrac{1}{2}, \\
\label{chicot}\int_0^xt^{\nu}I_{\nu}(t)\,\mathrm{d}t&=\sqrt{\pi}2^{\nu-1}\Gamma(\nu+\tfrac{1}{2})x(I_{\nu}(x)\mathbf{L}_{\nu-1}(x)-I_{\nu-1}(x)\mathbf{L}_{\nu}(x)), \quad \nu>-\tfrac{1}{2}.
\end{align}
The integral $\int_x^\infty t^{\nu}K_{\nu}(t)\,\mathrm{d}t$ can be bounded using the double inequality (\ref{lowerk2}).  Applying this bound to (\ref{chicot0}), rearranging and using that $I_{\nu,0}=\sqrt{\pi}\Gamma(\nu+\frac{1}{2})2^{\nu-1}$ yields the double inequality (\ref{dob11}).  To obtain (\ref{dob22}) we proceed similarly by bounding the integral $\int_0^xt^{\nu}I_{\nu}(t)\,\mathrm{d}t$ using inequalities (\ref{besi11}) and (\ref{besi22}) (with $n=0$).  The result now follows on using the formula (\ref{chicot}) and rearranging.
\end{proof}

%The double inequality (\ref{dob22}) is remarkably tight if either $\nu$ or $x$ are large. 

\begin{remark}We know from Theorem \ref{tiger} that the constants in the lower and upper bounds of (\ref{dob11}) are best possible.  The double inequality (\ref{dob22}) is tight in the limits $\nu\rightarrow\infty$ and $x\rightarrow\infty$.  To elaborate further, we denote $F_\nu(x)=I_{\nu}(x)\mathbf{L}_{\nu-1}(x)-I_{\nu-1}(x)\mathbf{L}_{\nu}(x)$ and denote the lower and upper bounds in (\ref{dob22}) by $L_\nu(x)$ and $U_\nu(x)$.  We now note the bound $\frac{I_{\nu+1}(x)}{I_\nu(x)}>\frac{x}{2(\nu+1)+x}$, $\nu>-1$, which is the simplest lower bound of a sequence of more complicated rational lower bounds given in \cite{nasell2}.  We thus obtain that the relative error in approximating $F_\nu(x)$ by either $L_\nu(x)$ or $U_\nu(x)$ is at most
\begin{align*}\frac{1}{2\nu+1}\bigg(1-\frac{I_{\nu+3}(x)}{I_{\nu+1}(x)}\bigg)&=\frac{1}{2\nu+1}\bigg(1-\frac{I_{\nu+3}(x)}{I_{\nu+2}(x)}\frac{I_{\nu+2}(x)}{I_{\nu+1}(x)}\bigg) \\
&<\frac{1}{2\nu+1}\bigg(1-\frac{x^2}{(2(\nu+3)+x)(2(\nu+2)+x)}\bigg)\\
&=\frac{4(\nu+2)(\nu+3)+(4\nu+10)x}{(2\nu+1)(2(\nu+2)+x)(2(\nu+3)+x)},
\end{align*}
which, for fixed $x$, has rate $\nu^{-1}$ as $\nu\rightarrow\infty$ and, for fixed $\nu$, has rate $x^{-1}$ as $x\rightarrow\infty$.  

We used Mathematica to compute the relative error in approximating $F_\nu(x)$ by $L_\nu(x)$ and $U_\nu(x)$, and numerical results are given in Tables \ref{table1} and \ref{table2}.  We observe that, for a given $x$, the relative error in approximating $F_\nu(x)$ by either $L_\nu(x)$ or $U_\nu(x)$ decreases as $\nu$ increases.  We also notice from Table \ref{table1} that, for a given $\nu$, the relative error in approximating $F_\nu(x)$ by $L_\nu(x)$ decreases as $x$ increases.  However, from Table \ref{table2} we see that, for a given $\nu$, as $x$ increases the relative error in approximating $F_\nu(x)$ by $U_\nu(x)$ initially increases before decreasing.  This is because, for $\nu>-\frac{1}{2}$, $\lim_{x\downarrow0}\frac{U_\nu(x)}{F_\nu(x)}=1$, meaning that the relative error in approximating $F_\nu(x)$ by $U_\nu(x)$ is 0 in the limit $x\downarrow0$.  The limit $\lim_{x\downarrow0}\frac{U_\nu(x)}{F_\nu(x)}=1$ follows from combining the formula $F_\nu(x)=\frac{1}{\sqrt{\pi}2^{\nu-1}\Gamma(\nu+\frac{1}{2})}x^{-1}\int_0^x t^\nu I_\nu(t)\,\mathrm{d}t$, the limit $\lim_{x\downarrow0}\frac{I_{\nu+3}(x)}{I_{\nu+1}(x)}=0$ and the limiting forms (\ref{form1}) and (\ref{form2}).

\begin{table}[h]
\begin{center}
\caption{\footnotesize{Relative error in approximating $F_\nu(x)$ by $L_\nu(x)$.}}
\label{table1}
{\scriptsize
\begin{tabular}{|c|rrrrrrr|}
\hline
%& & & & & $\nu$ & & & & \\
 \backslashbox{$\nu$}{$x$}      &    0.5 &    5 &    10 &    15 &    25 &    50 & 100   \\
 \hline
$-0.25$ & $0.6603$& $0.2881$ & $0.1124$ & $0.0709$ & $0.0414$ & $0.0203$ & 0.0101  \\
0 & $0.4948$& $0.2359$ & $0.1076$ & $0.0695$ & $0.0409$ & $0.0202$ & 0.0101  \\
2.5 & $0.1424$& $0.1129$ & $0.0776$ & $0.0570$ & $0.0366$ & $0.0192$ & 0.0098  \\
5 & $0.0832$& $0.0746$ & $0.0595$ & $0.0475$ &  $0.0329$  &  0.0182 & 0.0096 \\
7.5 & $0.0588$& $0.0552$ & $0.0476$ & $0.0403$ & $0.0296$ & $0.0173$ & 0.0093  \\ 
10 & $0.0454$& $0.0436$ & 0.0394  & 0.0346    &  0.0268     &  0.0164 & 0.0091 \\  
  \hline
\end{tabular}
}
\end{center}
\end{table}
\begin{table}[h]
\begin{center}
\caption{\footnotesize{Relative error in approximating $F_\nu(x)$ by $U_\nu(x)$.}}
\label{table2}
{\scriptsize
\begin{tabular}{|c|rrrrrrr|}
\hline
%& & & & & $\nu$ & & & & \\
 \backslashbox{$\nu$}{$x$}      &    0.5 &    5 &    10 &    15 &    25 &    50 & 100   \\
 \hline
$-0.25$ & $0.0103$& $0.4675$ & $0.4323$ & $0.3268$ & $0.2137$ & $0.1134$ & 0.0584  \\
0 & $0.0051$& $0.2038$ & $0.1973$ & $0.1543$ & $0.1034$ & $0.0558$ & 0.0290  \\
2.5 & $0.0001$& $0.0084$ & $0.0144$ & $0.0149$ & $0.0125$ & $0.0080$ & 0.0045  \\
5 & $0.0000$& $0.0017$ & $0.0039$ & $0.0049$ &  $0.0050$  &  0.0037 & 0.0023 \\ 
7.5 & $0.0000$& $0.0005$ & $0.0015$ & $0.0021$ &  $0.0025$  &  0.0022 & 0.0014 \\
10 & $0.0000$& $0.0002$ & 0.0007  & 0.0011    &  0.0015     &  0.0014 & 0.0010 \\  
  \hline
\end{tabular}
}
\end{center}
\end{table}
\end{remark}

From inequality (\ref{besi33}) of Theorem \ref{tiger1} we obtain the bound
\begin{equation}\label{bes18} \int_0^x \mathrm{e}^{-\gamma t}t^\nu I_\nu(t)\,\mathrm{d}t<\frac{2(\nu+1)}{(2\nu+1)(1-\gamma)}\mathrm{e}^{-\gamma x}x^\nu I_{\nu+1}(x), \quad \nu\geq \tfrac{1}{2},\:0<\gamma<1,
\end{equation}
which will be used in Section \ref{sec3} to bound the first expression in (\ref{mlhwsb2}) (note that the case $\nu>-\frac{1}{2}$, $\gamma\leq0$ is easily dealt with in Proposition \ref{propone}).  However, it would be desirable to obtain an analogue of (\ref{bes18}) that holds for all $\nu>-\frac{1}{2}$ and $0<\gamma<1$.  One difficulty in extending the parameter range to $\nu>-\frac{1}{2}$ and $0<\gamma<1$ is that the derivation of inequality (\ref{gauntpa}) (see parts (ii), (iii) and (iv) of Theorem 2.1 of \cite{gaunt ineq1}), which is used to prove inequality (\ref{besi33}), relies heavily on the inequality $I_{\nu}(x)<I_{\nu-1}(x)$ which only holds for all $x>0$ if $\nu\geq\frac{1}{2}$.  

In the following theorem, we make some progress towards extending the parameter range to $\nu>-\frac{1}{2}$ and $0<\gamma<1$, and then conclude this section by stating some open problems.  Before stating the theorem we shall introduce some notation.  Suppose $\nu>-\frac{1}{2}$.  Then we let $a_\nu$ be the largest number in the interval $[0,1]$ such that, for all $x>0$, 
\begin{equation}\label{anu}I_{\nu+1}(x)<(1-a_\nu)I_\nu(x)+a_\nu I_{\nu+2}(x).
\end{equation} 
That there exists such an $a_\nu$ in the interval $[0,1]$ can be seen from inequality (\ref{Imon}).  Inequality (\ref{anu}) is a useful refinement of the well-known inequality $I_{\nu+1}(x)<I_\nu(x)$, $\nu>-\frac{1}{2}$.  However, as far as the author is aware, the inequality  has not been studied in the literature.  A detailed analytic study of this inequality would go beyond the scope of this paper, but see Remark \ref{remopen} and Open Problem \ref{openprob2} below.  With this notation we may state our theorem.

\begin{theorem}\label{openthm}Suppose that $0<\gamma<\min\big\{\frac{1}{2(\nu+1)a_\nu},\frac{2\nu+1}{2(\nu+1)(1-a_\nu)}\big\}$.  Then, for all $x>0$,
\begin{equation}\label{op14}\int_0^x \mathrm{e}^{-\gamma t}t^\nu I_\nu(t)\,\mathrm{d}t<\frac{2(\nu+1)}{(2\nu+1)(1-(1-a_\nu)\gamma)-(1-a_\nu)\gamma}\mathrm{e}^{-\gamma x}x^\nu I_{\nu+1}(x). 
\end{equation}
\end{theorem}

\begin{proof}On using the differentiation formula (\ref{gafor}) we have that
\begin{equation*}\frac{\mathrm{d}}{\mathrm{d}t}\big(\mathrm{e}^{-\gamma t}t^\nu I_{\nu+1}(t)\big)=\frac{2\nu+1}{2(\nu+1)}\mathrm{e}^{-\gamma t}t^\nu I_{\nu}(t)+\frac{1}{2(\nu+1)}\mathrm{e}^{-\gamma t}t^\nu I_{\nu+2}(t)-\gamma\mathrm{e}^{-\gamma t}t^\nu I_{\nu+1}(t).
\end{equation*}
Integrating over $(0,x)$ and then rearranging gives that
\begin{align*}\int_0^x \mathrm{e}^{-\gamma t}t^\nu I_{\nu}(t)\,\mathrm{d}t&=\frac{2(\nu+1)}{2\nu+1}\mathrm{e}^{-\gamma x}x^\nu I_{\nu+1}(x)+\frac{2\gamma(\nu+1)}{2\nu+1}\int_0^x \mathrm{e}^{-\gamma t}t^\nu I_{\nu+1}(t)\,\mathrm{d}t\\
&\quad-\frac{1}{2\nu+1}\int_0^x \mathrm{e}^{-\gamma t}t^\nu I_{\nu+2}(t)\,\mathrm{d}t.
\end{align*}
We now apply inequality (\ref{anu}) to obtain
\begin{align}\int_0^x \mathrm{e}^{-\gamma t}t^\nu I_{\nu}(t)\,\mathrm{d}t&<\frac{2(\nu+1)}{2\nu+1}\mathrm{e}^{-\gamma x}x^\nu I_{\nu+1}(x)+\frac{2(1-a_\nu)\gamma(\nu+1)}{2\nu+1}\int_0^x \mathrm{e}^{-\gamma t}t^\nu I_{\nu}(t)\,\mathrm{d}t\nonumber\\
\label{aligned}&\quad-\frac{1-2a_\nu\gamma(\nu+1)}{2\nu+1}\int_0^x \mathrm{e}^{-\gamma t}t^\nu I_{\nu+2}(t)\,\mathrm{d}t \\
&<\frac{2(\nu+1)}{2\nu+1}\mathrm{e}^{-\gamma x}x^\nu I_{\nu+1}(x)+\frac{2(1-a_\nu)\gamma(\nu+1)}{2\nu+1}\int_0^x \mathrm{e}^{-\gamma t}t^\nu I_{\nu}(t)\,\mathrm{d}t. \nonumber
\end{align}
As $\gamma<\frac{1}{2(\nu+1)a_\nu}$ and $\gamma<\frac{2\nu+1}{2(\nu+1)(1-a_\nu)}$ we can rearrange to obtain
\begin{equation*}\int_0^x \mathrm{e}^{-\gamma t}t^\nu I_{\nu}(t)\,\mathrm{d}t<\frac{2(\nu+1)}{(2\nu+1)(1-(1-a_\nu)\gamma)-(1-a_\nu)\gamma}\mathrm{e}^{-\gamma x}x^\nu I_{\nu+1}(x),
\end{equation*}
as required.
\end{proof}

\begin{corollary}\label{cord} Suppose that $0<\gamma<\frac{2\nu+1}{2(\nu+1)}$.  Then, for all $x>0$,
\begin{equation}\label{op12}\int_0^x \mathrm{e}^{-\gamma t}t^\nu I_\nu(t)\,\mathrm{d}t<\frac{2(\nu+1)}{(2\nu+1)(1-\gamma)-\gamma}\mathrm{e}^{-\gamma x}x^\nu I_{\nu+1}(x).  
\end{equation}
\end{corollary}

\begin{proof}Let $a_\nu\uparrow1$ in Theorem \ref{openthm}.
\end{proof}

\begin{remark}\label{remopen}
\begin{enumerate}

\item The constant in (\ref{op14}) is larger than the constant in (\ref{bes18}).

\item We could obtain a refinement of inequality (\ref{op14}) by using inequality (\ref{besi11}) to lower bound the final integral of (\ref{aligned}), just as we did in deriving inequality (\ref{besi11}).  This refinement would, however, not be useful in extending the parameter range any further.

\item Determining the parameter range of validity in Corollary (\ref{cord}) is immediate, but a little more work is needed when working with Theorem \ref{openthm}.  For example, suppose we are interested in regime $\nu\geq0$.  Then it is immediate that inequality (\ref{op12}) is valid for $0<\gamma<\frac{1}{2}$.  One can numerically check that $a_0\approx0.25$.  Therefore, inequality (\ref{op14}) holds for $\nu\geq0$ and $0<\gamma<\min\{\frac{1}{2\cdot 0.25},\frac{1}{2(1-0.25)}\}=\frac{1}{2(1-0.25)}=0.66$.

\end{enumerate}
\end{remark}

We end this section by stating some open problems.  Open Problem \ref{openprob1} is of particular interest to this author because it would have a useful application to Stein's method for VG approximation.  Open Problem \ref{openprob2}  will most likely not be useful in solving Open Problem \ref{openprob1}, but as inequality (\ref{anu}) is a useful refinement of the classical inequality $I_{\nu+1}(x)<I_\nu(x)$, $\nu>-\frac{1}{2}$, it is considered to be interesting by this author.

\begin{open}\label{openprob1}Find a constant $C_{\nu,\beta}>0$ such that, for all $x>0$,
\begin{equation*}\label{bes1} \int_0^x \mathrm{e}^{-\gamma t}t^\nu I_\nu(t)\,\mathrm{d}t<C_{\nu,\beta}\mathrm{e}^{-\gamma x}x^\nu I_{\nu+1}(x), \quad \nu>-\tfrac{1}{2},\:0<\gamma<1.
\end{equation*}
\end{open}

\begin{open}\label{openprob2}Let $\nu>-\frac{1}{2}$.  Establish lower and upper bounds for $a_\nu$ that improve on the trivial estimate $0\leq a_\nu\leq1$.
\end{open}

There is of course an analogous problem for the modified Bessel function $K_\nu(x)$.  Let $\nu>-\frac{1}{2}$.  Then let $b_v$ be the largest number in the interval $[0,1]$ such that, for all $x>0$,
\begin{equation*}K_{\nu+1}(x)<b_\nu K_\nu(x)+(1-b_\nu) K_{\nu+2}(x).
\end{equation*} 

\begin{open}Let $\nu>-\frac{1}{2}$.  Establish lower and upper bounds for $b_\nu$ that improve on the trivial estimate $0\leq b_\nu\leq1$.
\end{open}

\section{Uniform bounds for expressions involving integrals of modified Bessel functions}\label{sec3}

In this section, we use the integral inequalities of Section \ref{sec2} and straightforward calculations to obtain uniform bounds for expressions of the type that were presented in the Introduction.  These bounds are required for technical advances in Stein's method for VG approximation \cite{gaunt vg2, gaunt vg3}. Before stating these bounds, we collect some useful inequalities for products of modified Bessel functions in the following lemma.  Part (i) is  given in the proof of Theorem 5 of \cite{gaunt ineq2}, and is a simple consequence of Theorem 4.1 of \cite{hartman}. Part (ii) is proved in Lemma 3 of \cite{gaunt ineq2}.  See also \cite{baricz} and \cite{baricz2} for a number of results and upper bounds for the product $I_{\nu}(x)K_{\nu}(x)$.

\begin{lemma}\label{IKineq}%(i) Let $\nu>0$. Then, for all $x \ge 0$, 
%\begin{equation} \label{mice} 0<K_{\nu}(x) I_{\nu}(x) \leq \frac{1}{2\nu}.
%\end{equation}
(i) Let $\nu>\frac{1}{2}$. Then, for all $x\geq0$, 
\begin{equation}\label{bdsjbc1}0\leq xK_{\nu}(x)I_\nu(x)<\frac{1}{2}.
\end{equation}

(ii) Fix $\nu\geq-\frac{1}{2}$.  Then, for all $x\geq0$, 
\begin{equation}\label{bdsjbc}\frac{1}{2}< xK_{\nu+1}(x)I_\nu(x)\leq1.
\end{equation}
\end{lemma}

With the aid of this lemma and the results of Section \ref{sec2} we may prove the following theorems.

\begin{theorem}\label{notfin}(i) Let $\beta\geq0$ and $\nu>-\frac{1}{2}$.  Then, for all $x\geq0$,
\begin{eqnarray}\label{propb2a12}\frac{\mathrm{e}^{-\beta x}K_{\nu+1}(x)}{x^\nu}\int_0^x \mathrm{e}^{\beta t}t^{\nu+1}I_\nu(t)\,\mathrm{d}t&<& \frac{1}{2}, \\
\label{propb2a125}\frac{\mathrm{e}^{-\beta x}K_{\nu}(x)}{x^\nu}\int_0^x \mathrm{e}^{\beta t}t^{\nu+1}I_\nu(t)\,\mathrm{d}t&<& \frac{1}{2},
\end{eqnarray}
(ii) and
\begin{eqnarray}\label{jjj1}\frac{\mathrm{e}^{-\beta x}K_{\nu+1}(x)}{x^{\nu-1}}\int_0^x \mathrm{e}^{\beta t}t^{\nu}I_\nu(t)\,\mathrm{d}t&<& \frac{\nu+1}{2\nu+1}, \\
\label{jjj2}\frac{\mathrm{e}^{-\beta x}K_{\nu}(x)}{x^{\nu-1}}\int_0^x \mathrm{e}^{\beta t}t^{\nu}I_\nu(t)\,\mathrm{d}t&<& \frac{\nu+1}{2\nu+1}.
\end{eqnarray}
(iii) Let $-1<\beta<0$ and $\nu>-\frac{1}{2}$.  Then, for all $x\geq0$,
\begin{eqnarray*}\frac{\mathrm{e}^{-\beta x}K_{\nu+1}(x)}{x^\nu}\int_0^x \mathrm{e}^{\beta t}t^{\nu+1}I_\nu(t)\,\mathrm{d}t&<& \frac{1}{2(1+\beta)}, \\
\frac{\mathrm{e}^{-\beta x}K_{\nu}(x)}{x^\nu}\int_0^x \mathrm{e}^{\beta t}t^{\nu+1}I_\nu(t)\,\mathrm{d}t&<& \frac{1}{2(1+\beta)}.
\end{eqnarray*}
(iv) Now suppose that $-1<\beta<0$ and $\nu\geq\frac{1}{2}$.  Then, for all $x\geq0$,
\begin{eqnarray*}\frac{\mathrm{e}^{-\beta x}K_{\nu+1}(x)}{x^{\nu-1}}\int_0^x \mathrm{e}^{\beta t}t^{\nu}I_\nu(t)\,\mathrm{d}t&<& \frac{\nu+1}{(2\nu+1)(1+\beta)}, \\
\frac{\mathrm{e}^{-\beta x}K_{\nu}(x)}{x^{\nu-1}}\int_0^x \mathrm{e}^{\beta t}t^{\nu}I_\nu(t)\,\mathrm{d}t&<& \frac{\nu+1}{(2\nu+1)(1+\beta)}.
\end{eqnarray*}
\end{theorem}

\begin{proof}(i) Suppose $\beta\geq0$ and $\nu>-\frac{1}{2}$. We have
\begin{align*}\frac{\mathrm{e}^{-\beta x}K_{\nu+1}(x)}{x^\nu}\int_0^x \mathrm{e}^{\beta t}t^{\nu+1}I_\nu(t)\,\mathrm{d}t&\leq \frac{\mathrm{e}^{-\beta x}K_{\nu+1}(x)}{x^\nu}\cdot \mathrm{e}^{\beta x}x^{\nu+1}I_{\nu+1}(x)\\
& =xK_{\nu+1}(x)I_{\nu+1}(x)<\frac{1}{2},
\end{align*} 
where we used (\ref{propb2a}) to obtain the first inequality and (\ref{bdsjbc1}) to obtain the second inequality, which proves (\ref{propb2a12}).  We obtain (\ref{propb2a125}) from (\ref{propb2a12}) by an application of inequality (\ref{cake}).

(ii) Applying inequalities (\ref{propb2a1}) and (\ref{bdsjbc1}) gives that
\begin{align*}\frac{\mathrm{e}^{-\beta x}K_{\nu+1}(x)}{x^{\nu-1}}\int_0^x \mathrm{e}^{\beta t}t^{\nu}I_\nu(t)\,\mathrm{d}t&\leq \frac{\mathrm{e}^{-\beta x}K_{\nu+1}(x)}{x^{\nu-1}}\cdot \frac{2(\nu+1)}{2\nu+1}\mathrm{e}^{\beta x}x^\nu I_{\nu+1}(x) \\
&=\frac{2(\nu+1)}{2\nu+1}xK_{\nu+1}(x)I_{\nu+1}(x)<\frac{\nu+1}{2\nu+1},
\end{align*}
which proves (\ref{jjj1}).  We deduce inequality (\ref{jjj2}) from (\ref{jjj1}) by applying inequality (\ref{cake}).

(iii) The argument is the same as for part (i), but, since now $-1<\beta<0$, we use inequality (\ref{besi55}) to bound the integral instead of inequality (\ref{propb2a}).

(iv) The proof is the same as for part (iii), except now we use (\ref{bes18}) to bound the integral, instead of (\ref{propb2a1}).
\end{proof}

\begin{theorem}(i) Let $\beta\leq0$ and $\nu>-\frac{1}{2}$.  Then, for all $x\geq0$,
\begin{eqnarray}\label{fff11}\frac{\mathrm{e}^{-\beta x}I_{\nu}(x)}{x^\nu}\int_x^\infty \mathrm{e}^{\beta t}t^{\nu+1}K_{\nu}(t)\,\mathrm{d}t&<& 1, \\
\label{fff2}\frac{\mathrm{e}^{-\beta x}I_{\nu+1}(x)}{x^\nu}\int_x^\infty \mathrm{e}^{\beta t}t^{\nu+1}K_{\nu}(t)\,\mathrm{d}t&<& \frac{1}{2},
\end{eqnarray}
(ii) and
\begin{eqnarray}\label{ddd1}\frac{\mathrm{e}^{-\beta x}I_{\nu}(x)}{x^{\nu-1}}\int_x^\infty \mathrm{e}^{\beta t}t^{\nu}K_{\nu}(t)\,\mathrm{d}t&<& 1, \\
\label{ddd2}\frac{\mathrm{e}^{-\beta x}I_{\nu+1}(x)}{x^{\nu-1}}\int_x^\infty \mathrm{e}^{\beta t}t^{\nu}K_{\nu}(t)\,\mathrm{d}t&<& \frac{1}{2}.
\end{eqnarray}
(iii) Let $0< \beta <1$ and $\nu>-\frac{1}{2}$. Then, for all $x\geq0$, 
\begin{eqnarray*}\frac{\mathrm{e}^{-\beta x}I_{\nu}(x)}{x^\nu}\int_x^\infty \mathrm{e}^{\beta t}t^{\nu+1}K_{\nu}(t)\,\mathrm{d}t&<& 1+\frac{2\sqrt{\pi}\beta\Gamma(\nu+\frac{3}{2})}{(1-\beta^2)^{\nu+\frac{3}{2}}\Gamma(\nu+1)}, \\
\frac{\mathrm{e}^{-\beta x}I_{\nu+1}(x)}{x^\nu}\int_x^\infty \mathrm{e}^{\beta t}t^{\nu+1}K_{\nu}(t)\,\mathrm{d}t&<& \frac{1}{2}+\frac{\sqrt{\pi}\beta\Gamma(\nu+\frac{3}{2})}{(1-\beta^2)^{\nu+\frac{3}{2}}\Gamma(\nu+1)},
\end{eqnarray*}
(iv) and
\begin{eqnarray}\label{ddd3}\frac{\mathrm{e}^{-\beta x}I_{\nu}(x)}{x^{\nu-1}}\int_x^\infty \mathrm{e}^{\beta t}t^{\nu}K_{\nu}(t)\,\mathrm{d}t&<& N_{\nu,\beta}, \\
\label{ddd4}\frac{\mathrm{e}^{-\beta x}I_{\nu+1}(x)}{x^{\nu-1}}\int_x^\infty \mathrm{e}^{\beta t}t^{\nu}K_{\nu}(t)\,\mathrm{d}t&<& N_{\nu,\beta},
\end{eqnarray}
where 
\begin{equation*}N_{\nu,\beta}= \begin{cases}\displaystyle \frac{1}{2(1-\beta)}, & \quad \nu\leq\tfrac{1}{2}, \\ 
\displaystyle\frac{\sqrt{\pi}\Gamma(\nu+\frac{1}{2})}{(1-\beta^2)^{\nu+\frac{1}{2}}\Gamma(\nu)}, & \quad \nu>\tfrac{1}{2}. \end{cases}
\end{equation*}
\end{theorem}

\begin{proof} (i) Let us first prove inequality (\ref{fff11}). We have
\begin{align*}\frac{\mathrm{e}^{-\beta x}I_{\nu}(x)}{x^\nu}\int_x^\infty \mathrm{e}^{\beta t}t^{\nu+1}K_{\nu}(t)\,\mathrm{d}t&< \frac{\mathrm{e}^{-\beta x}I_{\nu}(x)}{x^\nu}\cdot \mathrm{e}^{\beta x}x^{\nu+1}K_{\nu+1}(x) \\
&=xI_\nu(x)K_{\nu+1}(x)\leq 1,
\end{align*}
where we used (\ref{fff}) to obtain the first inequality and (\ref{bdsjbc}) to obtain the second inequality.  To obtain inequality (\ref{fff2}), we apply inequality (\ref{bdsjbc1}) instead of inequality (\ref{bdsjbc}):
\begin{align*}\frac{\mathrm{e}^{-\beta x}I_{\nu+1}(x)}{x^\nu}\int_x^\infty \mathrm{e}^{\beta t}t^{\nu+1}K_{\nu}(t)\,\mathrm{d}t< xI_{\nu+1}(x)K_{\nu+1}(x)<\frac{1}{2}.
\end{align*}

(ii) Using inequalities (\ref{fff1}) and (\ref{bdsjbc}) gives that
\begin{align*}\frac{\mathrm{e}^{-\beta x}I_{\nu}(x)}{x^{\nu-1}}\int_x^\infty \mathrm{e}^{\beta t}t^{\nu}K_\nu(t)\,\mathrm{d}t< \frac{\mathrm{e}^{-\beta x}I_{\nu}(x)}{x^{\nu-1}}\cdot \mathrm{e}^{\beta x} x^\nu K_{\nu+1}(x) =xI_\nu(x)K_{\nu+1}(x)\leq1,
\end{align*}
which proves (\ref{ddd1}).  We establish (\ref{ddd2}) similarly, but use inequality (\ref{bdsjbc1}) to bound the product of modified Bessel functions instead of inequality (\ref{bdsjbc}).

(iii) As was the case for the proof of Theorem \ref{notfin}, the argument is the same as for part (i), but, since now $0<\beta<1$, we use inequality (\ref{lowerk3}) to bound the integral.  We obtain a final simplification by using the upper bound in (\ref{doubleivb}) to bound $I_{\nu+1,\beta}$.

(iv) If $0<\beta<1$ and $\nu\leq\frac{1}{2}$, then applying inequalities (\ref{lowerk}) and (\ref{bdsjbc1}) gives that
\begin{align*}\frac{\mathrm{e}^{-\beta x}I_{\nu}(x)}{x^{\nu-1}}\int_x^\infty \mathrm{e}^{\beta t}t^{\nu}K_\nu(t)\,\mathrm{d}t&\leq \frac{\mathrm{e}^{-\beta x}I_{\nu}(x)}{x^{\nu-1}}\cdot\frac{1}{1-\beta}\mathrm{e}^{\beta x}x^\nu K_\nu(x) \\
&=\frac{1}{1-\beta}xI_\nu(x)K_\nu(x)<\frac{1}{2(1-\beta)},
\end{align*}
as required.  If now $\nu>\frac{1}{2}$, then we argue as before but use (\ref{lowerk2}) instead of (\ref{lowerk}).  We obtain a final simplification by using the upper bound in (\ref{doubleivb}) to bound $I_{\nu,\beta}$.   This completes the proof of (\ref{ddd3}), and we deduce (\ref{ddd4}) by applying inequality (\ref{Imon}).
\end{proof}

\begin{remark}We could use Theorem \ref{openthm} to obtain uniform bounds on the expressions of part (iv) of Theorem \ref{notfin} that hold in a larger parameter regime.
\end{remark}

All the theorems in this section give uniform bounds for the expressions involving integrals of modified Bessel functions in the entire parameter range $\nu>-\frac{1}{2}$, $-1<\beta<1$, except for part (iv) of Theorem \ref{notfin}.  Achieving this would have a useful application to Stein's method for VG approximation.  It should be noted that a straightforward asymptotic analysis verifies that the expression is bounded for all $x\geq0$, and it remains to find an explicit upper bound in terms of $\nu$ and $\beta$.  If Open Problem \ref{openprob1} was to be solved then one could easily obtain an explicit upper bound by the arguments used in this section.  The author considers this to be the most promising approach, but working directly with the integral offers an alternative.

\begin{open}\label{openprob3}Find a constant $C_{\nu,\beta}>0$ such that, for all $x\geq0$,
\begin{equation*}\frac{\mathrm{e}^{-\beta x}K_{\nu+1}(x)}{x^{\nu-1}}\int_0^x \mathrm{e}^{\beta t}t^{\nu}I_\nu(t)\,\mathrm{d}t\leq C_{\nu,\beta}, \quad \nu>-\tfrac{1}{2}, \: -1<\beta<0. 
\end{equation*}
\end{open}

\appendix

\section{Elementary properties of modified Bessel functions}

Here we list standard properties of modified Bessel functions that are used throughout this paper.  All these formulas can be found in \cite{olver}, except for the inequalities and the integration formula (\ref{pdfk}), which can be found in \cite{gradshetyn}.

\subsection{Basic properties}
The modified Bessel functions $I_{\nu}(x)$ and $K_{\nu}(x)$ are both regular functions of $x\in\mathbb{R}$.  For positive values of $x$ the functions $I_{\nu}(x)$ and $K_{\nu}(x)$ are positive for $\nu>-1$ and all $\nu\in\mathbb{R}$, respectively.  For all $\nu\in\mathbb{R}$, $K_{-\nu}(x)=K_{\nu}(x)$.  The modified Bessel function $I_{\nu}(x)$ satisfies the identity
\begin{equation}\label{Iidentity}I_{\nu +1} (x) = I_{\nu -1} (x) - \frac{2\nu}{x} I_{\nu} (x).
\end{equation}

\subsection{Limiting forms}\label{asymsec}
\begin{align}\label{Itend0}I_{\nu} (x) &\sim \frac{1}{\Gamma(\nu +1)} \bigg(\frac{x}{2}\bigg)^{\nu}, \quad x \downarrow 0, \: \nu>-1, \\
\label{Itendinfinity}I_{\nu} (x) &\sim \frac{\mathrm{e}^{x}}{\sqrt{2\pi x}}, \quad x \rightarrow\infty, \: \nu\in\mathbb{R}, \\
\label{Ktend0}K_{\nu} (x) &\sim  2^{\nu -1} \Gamma (\nu) x^{-\nu},  \quad x\downarrow0, \: \nu>0, \\
\label{Ktendinfinity} K_{\nu} (x) &\sim \sqrt{\frac{\pi}{2x}} \mathrm{e}^{-x}, \quad x \rightarrow \infty, \: \nu\in\mathbb{R}, 
\end{align}
\begin{align}
\label{struve0}\mathbf{L}_{\nu}(x)&\sim \frac{2}{\sqrt{\pi}\Gamma(\nu+\frac{3}{2})}\bigg(\frac{x}{2}\bigg)^{\nu+1}, \quad x \downarrow 0, \: \nu>-\tfrac{1}{2}, \\
\label{struveinfinity}\mathbf{L}_{\nu}(x)&\sim \frac{\mathrm{e}^{x}}{\sqrt{2\pi x}}, \quad x \rightarrow\infty, \: \nu\in\mathbb{R}.
\end{align}

\subsection{Inequalities}
Let $x > 0$. Then the following inequalities hold
\begin{align}\label{Imon}I_{\nu} (x) < I_{\nu - 1} (x), \quad \nu \geq \tfrac{1}{2},\\
\label{Kmoni}K_{\nu} (x) < K_{\nu - 1} (x), \quad \nu < \tfrac{1}{2},\\
\label{cake}K_{\nu} (x) \geq K_{\nu - 1} (x), \quad \nu \geq \tfrac{1}{2}.  
\end{align}
We have equality in (\ref{cake}) if and only if $\nu=\frac{1}{2}$.  The inequalities for $K_{\nu}(x)$ can be found in \cite{ifantis}, whilst the inequality for $I_{\nu}(x)$ can be found in \cite{jones} and \cite{nasell}, which extends a result of \cite{soni}.  A survey of related inequalities for modified Bessel functions is given by  \cite{baricz2}, and lower and upper bounds for the ratios $\frac{I_{\nu}(x)}{I_{\nu-1}(x)}$ and $\frac{K_{\nu}(x)}{K_{\nu-1}(x)}$, which improve on inequalities (\ref{Imon}) -- (\ref{cake}), are also given in \cite{ifantis} and \cite{segura}.

\subsection{Differentiation}
\begin{align}\label{diffone}\frac{\mathrm{d}}{\mathrm{d}x} (x^{\nu} I_{\nu} (x) ) &= x^{\nu} I_{\nu -1} (x), \\
\label{diffKi}\frac{\mathrm{d}}{\mathrm{d}x} (x^{\nu} K_{\nu} (x) ) &= -x^{\nu} K_{\nu -1} (x). 
\end{align}

\subsection{Integration}
\begin{equation} \label{pdfk} \int_{-\infty}^{\infty}\mathrm{e}^{\beta t} |t|^{\nu} K_{\nu}(|t|)\,\mathrm{d}t =\frac{\sqrt{\pi}\Gamma(\nu+\frac{1}{2})2^{\nu}}{(1-\beta^2)^{\nu+\frac{1}{2}}}, \quad \nu>-\tfrac{1}{2}, \: -1<\beta <1.
\end{equation}

\subsection*{Acknowledgements}
The author is supported by a Dame Kathleen Ollerenshaw Research Fellowship.  The author would like to thank \'{A}rp\'{a}d Baricz and Tibor Pog\'any for stimulating discussions related to this work.  The author would like to thank the reviewer for their comments and suggestions which led to an improved paper.

\end{document}